\documentclass[reqno,11pt]{amsart}

\usepackage{amsfonts, amssymb, verbatim}
\usepackage{amssymb,epsfig,amscd, verbatim}
\usepackage{amsthm}
\usepackage{amsmath}
\usepackage{accents}
\usepackage[T1]{fontenc}
\usepackage{graphicx}
\usepackage{latexsym}

\newtheorem{theorem}{Theorem}

\newtheorem{proposition}[theorem]{Proposition}

\theoremstyle{definition}
\newtheorem{definition}[theorem]{Definition}

\numberwithin{equation}{section} \numberwithin{theorem}{section}

\theoremstyle{remark}

\def\z{\,_{\dot z}\,}
\def\w{\,_{\dot w}\,}
\def\zw{\,_{\dot {z+w}}\,}

\def\vac{|0\rangle}                            

\newcommand\lbb[1]{\label{#1}}

\def\<{\langle}
\def\>{\rangle}

\def\vac{\mathbf{1}}                            
\def\one{{\mathbf{1}}}

\newcommand{\kk}{\mathbf{k}}

\newcommand{\NN}{\mathbb{N}}

\newcommand{\ZZ}{\mathbb{Z}}

\def\de{\delta}

\def\z{\,_{\dot z}\,}
\def\x{\,_{\dot x}\,}
\def\y{\,_{\dot y}\,}
\def\w{\,_{\dot w}\,}

\def\mx{\,_{\dot {-x\ }}\,}
\def\mz{\,_{\dot {-z\ }}\,}
\def\zx{\,_{\dot {z+x}}\,}
\def\zy{\,_{\dot {z+y}}\,}
\def\yz{\,_{\dot {y+z}}\,}
\def\xz{\,_{\dot {x+z}}\,}

\def\mzmx{\,_{\dot{-z-x}}\,}

\def\noi{\noindent}

\def\Mz{\overset{{\scalebox{.64}{$M$}}\hskip .01cm}{\z}}
\def\Mzw{\overset{{\scalebox{.64}{$M$}}\hskip .01cm}{\zw}}

\def\VMz{\overset{{\scalebox{.64}{$V\oplus M$}}\hskip .01cm}{\z}}
\def\VMmz{\overset{{\scalebox{.64}{$V\oplus M$}}\hskip .01cm}{\mz}}
\def\VMx{\overset{{\scalebox{.64}{$V\oplus M$}}\hskip .01cm}{\x}}
\def\VMxz{\overset{{\scalebox{.64}{$V\oplus M$}}\hskip .01cm}{\xz}}

\def\sz{\underaccent{\ \  z}{\ast}}
\def\smz{\underaccent{\ \  -z}{\ast}}
\def\sw{\underaccent{\ \  w}{\ast}}
\def\szw{\underaccent{\ \  z+w}{\ast}}

\def\szz{\underaccent{\, z}{\ast}}

\def\szu{\overset{{\scalebox{.64}{(1)}}\hskip .01cm}{\underaccent{\,  z}{\ast}}}
\def\szd{\overset{{\scalebox{.64}{(2)}}\hskip .01cm}{\underaccent{\,  z}{\ast}}}

\def\zu{\overset{{\scalebox{.64}{(1)}}\hskip .01cm}{\z}}
\def\zd{\overset{{\scalebox{.64}{(2)}}\hskip .01cm}{\z}}

\def\Mw{\overset{{\scalebox{.64}{$M$}}\hskip .01cm}{\w}}

\def\Mzw{\overset{{\scalebox{.64}{$M$}}\hskip .01cm}{\zw}}

\def\Mz{\overset{{\scalebox{.64}{$M$}}\hskip .01cm}{\z}}
\def\Mz{\overset{{\scalebox{.64}{$M$}}\hskip .01cm}{\z}}

\def\ddz{\frac{d}{dz}\,}

\def\VA{(V,\z,\vac,D)}
\def\v2a{(V,\z,\vac,D)}
\begin{document}

\title{Cohomology of vertex algebras}

\author[Jos\'e I. Liberati]{Jos\'e I. Liberati$^*$}
\thanks {\textit{$^{*}$Ciem - CONICET (FAMAF), Medina Allende y
Haya de la Torre, Ciudad Universitaria, (5000) C\'ordoba,
Argentina.
e-mail: joseliberati@gmail.com}}
\address{{\textit{Ciem - CONICET (FAMAF), Medina Allende y
Haya de la Torre, Ciudad Universitaria, (5000) C\'ordoba -
Argentina.  E-mail: joseliberati@gmail.com}}}

\date{March 25, 2016}

\maketitle


\begin{abstract}
Let $V$ be a vertex algebra and $M$ a $V$-module. We define the
first and second cohomology of $V$ with coefficients in $M$, and
we show that the second cohomology $H^{2}(V, M)$  corresponds
bijectively to the set of equivalence classes of square-zero
extensions of $V$ by $M$. In the case that $M=V$, we show that the
second cohomology $H^{2}(V, V)$ corresponds bijectively to the set
of equivalence classes of first order deformations of $V$.
\end{abstract}


\section{Introduction}\lbb{intro}

\

This is the second of a series of papers (see \cite{L}) trying to
extend certain restricted definitions and constructions developed
for vertex operator algebras to the general framework of vertex
algebras without assuming any grading condition neither on the
vertex algebra nor on the modules involved, and we make a strong
emphasis on the commutative associative algebra point of view
instead of the  Lie theoretical point of view.

In this work we define the first and second cohomology of a vertex
algebra $V$ with coefficients in a $V$-module $M$, and we show
that the second cohomology $H^{2}(V, M)$ corresponds bijectively
to the set of equivalence classes of square-zero extensions of $V$
by $M$. In the case that $M=V$, we show that the second cohomology
$H^{2}(V, V)$ corresponds bijectively to the set of equivalence
classes of first order deformations of $V$. If we restrict it to a
vertex algebra given by an associative commutative algebra, then
we clearly obtain the Harrison's cohomology.

In \cite{H1,H2}, Huang developed the cohomology theory of graded
vertex algebras using analy- tical methods and complex variables. In
the present paper we develop the cohomology theory for vertex
algebras (without grading conditions) using algebraic methods and
formal variables, obtaining a very simplified, clear and nice
theory.

In the definition of $H^2(V,M)$,  Huang used two complex
variables. In fact, in the proofs of the theorems that relate the
second cohomology with extensions and deformations \cite{H2},
Huang passed from two complex variables to one formal variable. We
directly use one formal variable (cf. \cite{BKV} versus
\cite{DK}). So, if we add grading conditions to our definitions
and constructions, then we obtain a simpler algebraic version of
the results in \cite{H2}.

We keep using the more comfortable notation introduced in
\cite{L}, where the map $Y(a,z)b$ is replaced by $a\z b$.

This paper is organized as follows. In section 2, we introduce the
basic definitions and notations. In section 3, we define the first
and second cohomology of a vertex algebra $V$ with coefficients in
a module  $M$. In section 4, we show that the second cohomology
$H^{2}(V, M)$ corresponds bijectively to the set of equivalence
classes of square-zero extensions of $V$ by $M$. In section 5, in
the case that $M=V$, we show that the second cohomology $H^{2}(V,
V)$ corresponds bijectively to the set of equivalence classes of
first order deformations of $V$.

Unless otherwise specified, all vector spaces, linear maps and
tensor products are considered over an algebraically closed field
$\kk$ of characteristic 0.

\

\vskip .5cm

\section{Definitions and notation}\lbb{def }

\

In order to make a self-contained paper, in  this section we
present the notion of  vertex algebra and their modules. Our
presentation and notation differ from the usual one because we
want to emphasize the point of view that vertex algebras are
analog to commutative associative algebras with unit.

\

Throughout this work, we define $(x + y)^n$ for $n\in\ZZ$ (in
particular, for $n < 0$) to be the formal series
\begin{align*}
(x+y)^n=\sum_{k \in \mathbb{Z}_+} \binom{n}{k} x^{n-k}y^k,
\end{align*}
where $\binom{n}{k}=\frac{n(n-1)...(n-k+1)}{k!}$.

\

\definition \label{def VA}
A {\it vertex algebra} is a quadruple $\VA$ that consists of a
vector space $V$ equipped with a linear map
 \begin{align}\label{zzzz}
 \z \, :V &\otimes V \longrightarrow V((z))\\
 a &\otimes b \, \, \,  \longmapsto a\z b, \nonumber
 \end{align}
a distinguished vector  $\vac$ and $D\in\,$End$(V)$
satisfying the following axioms ($a,b,c \in  V):$\\
\vskip.1cm \noindent$\bullet$ {\it Unit}:
\begin{align*}
\vac \z a= a \quad  \textrm{ and }& \quad a \z
\vac=e^{zD}a;
\end{align*}

\noindent $\bullet$ {\it Translation - Derivation}: \vskip .1cm
\begin{align*}
(Da)\z b= \ddz (a\z b), \quad D(a\z b)=(Da)\z b + a\z (Db);
\end{align*}
\vskip .3cm

\noindent$\bullet$ {\it Commutativity}:
\begin{align*}
a \z b= e^{zD} (b\mz a);
\end{align*}
\vskip .3cm

\noindent$\bullet$  {\it Associativity}: For any $a,b,c\in V$,
there exist $l\in \NN$ such that

\begin{align} \label{associatt}
(z+w)^l\, \ (a \z b) \w c = (z+w)^l\, \  a\zw (b\w c) .
\end{align}
\vskip.3cm

\

Observe that the standard notation $Y(a,z)\, b$ for the
$z$-product in (\ref{zzzz}) has been changed. We adopted this
notation following the practical idea  of the  $\lambda$-bracket
in the notion of Lie conformal algebra (also called vertex Lie
algebra in the literature), see \cite{K}.

The commutativity axiom is known in the literature as
skew-symmetry (see \cite{LL, K}), but for us it really corresponds
to commutativity. We want to emphasize the point of view of a
vertex algebra as a generalization of an associative commutative
algebra with unit (and a derivation), as in \cite{BK, L1}, and
having in mind the usual (trivial) example of holomorphic vertex
algebra given by any associative commutative algebra $A$ with unit
$\one$ and a derivation $D$, with the $\z$-product defined by $a\z
b:=(e^{zD} a)\cdot b$. In the special case of $D=0$, we simply
have an associative commutative algebra with unit and $a\z b:= a
\cdot b$, and with the definition presented before, it is
immediate to see, without any computation, that this is a vertex
algebra, and this is the reason for us to call the usual
skew-symmetry axiom as commutativity.

An equivalent definition can be obtained by replacing the
associativity axiom by the {\it associator formula} (which is
equivalent to what is  known in the literature as the iterate
formula (see \cite{LL}, p.54-55) or the $n$-product identity (see
\cite{BK})):
\begin{align} \label{associator}
(a \z b) \y c- a\zy (b\y c)=  b\y (a\yz c - a\zy c),
\end{align}
for $a,b,c\in V$. Observe that in the last term of the associator
formula we can not use linearity to write it as a difference of $\
b\,_{\dot w}\big(\,a\,_{\dot {w+z}}\,c\big)$ and $\ b\,_{\dot
w}\big(\,a\,_{\dot {z+w}}\,c\big)$, because neither of these
expressions in general exists (see \cite{LL}, p.55). This
alternative definition of vertex algebra using the (iterate
formula or) the associator formula is essentially the original
definition given by Borcherds \cite{Bo1}, but in our case it is
written using the generating series in $z$ instead of the
$n$-products.

It is well known (see \cite{LL}) that in the two equivalent
definitions that we presented, some of the axioms can  be obtained
from the others, but we prefer to make
 emphasis    on the properties of $D$ and the explicit formula for the multiplication
 by the unit.

\

\begin{definition} A {\it module} over a vertex algebra $V$ is
a vector space $M$ equipped  with an endomorphism $d$ of $M$ and a
linear map
    \begin{align*}
      &V\otimes M \longrightarrow M((z))  \\
      &(a,u)\ \longmapsto \ a\, \Mz \, u
    \end{align*}
satisfying the following axioms ($a \in  V$ and $u\in M$):
\\
\vskip.1cm \noindent$\bullet$ {\it Unit}:
\begin{align*}
\label{1} \vac \Mz u= u;
\end{align*}
$\bullet$ {\it Translation - Derivation}:
\begin{align*}
(Da)\Mz u= \ddz (a\Mz u), \quad d(a\Mz u)=(Da)\Mz u + a\Mz (d\,u);
\end{align*}
\vskip.2cm \noindent $\bullet$  {\it Associativity}:  For any
$a,b\in V$ and $u\in M$, there exist $l\in\NN$ such that

\begin{equation}\label{aaaa}
(z+w)^l\, \ (\,a\, _{\dot z} \, b\, )\Mw u\ =(z+w)^l\ \,a \Mzw (\, b\Mw u\,)\ .
\end{equation}

\end{definition}

\

\vskip .2cm

Sometimes, if everything is clear, we shall use $a\z u$ instead of
$a \Mz u$. Obviously, $V$ is a module over $V$. We follow
\cite{BK}, in the definition of module, because we need to work
with this $\kk[d]$-module structure
 (similar to the situation of  Lie conformal algebras \cite{K}).

Any $V$-module $M$ satisfies the {\it weak commutativity} or {\it
locality}: for all $a,b\in V$, there exist $k\in\NN$ such that
\begin{equation}\label{locality}
    (x-y)^k \, a\x (b\y u)= (x-y)^k \,b\y (a\x u) \quad  \hbox{for all }u\in M.
\end{equation}

Let $M$ and $N$ be $V$-modules, a {\it $V$-homomorphism} or a {\it
homomorphism of $V$-modules} from $M$ to $N$ is a linear map
$\varphi:M\to N$ such that for $a\in V$ and $u\in M$
$$
\varphi(a\z u)= a\z \varphi(u) \ \ \mathrm{ and }\ \ \varphi(d u)=
d \varphi(u) .
$$

A subspace $W$ of a vertex algebra $V$ is called an {\it ideal of}
$V$ if $a\z b\in W$ for all $a\in V$ and $b\in W$.

\

\vskip .5cm


\section{Definition of  lower cohomologies}\lbb{first}


\

Let $V$ be a vertex algebra, and $M$ a (left) $V$-module. We
define the {\it right action} of $V$ on $M$ by
\begin{equation*}\label{}
m\z a=e^{zd}(a\mz m),
\end{equation*}

\vskip .3cm

\noi for $a\in V$ and $m\in M$. A linear map $f:V\rightarrow M$ is
called a {\it vertex derivation} if

\begin{equation*}
f(a\z b)=a\z f(b)+ f(a)\z b
\end{equation*}

\vskip .3cm

\noi for $a,b\in V$. We denote by VDer$(V,M)$ the space of all
such derivations.

Now, we consider
$$
0\xrightarrow{\quad\ \
\quad}C^0(V,M)\xrightarrow{\quad\de_0\quad}C^1(V,M)\xrightarrow{\quad\de_1\quad}C^2(V,M)
$$
\vskip .3cm

\noi with $C^0(V,M)=M$ and $\de_0\equiv 0$, therefore
$H^0(V,M)=$Ker $\de_0=M$. We define

\begin{equation*}
C^1(V,M)=\{ \ g\in {\rm Hom}_\kk (V,M) \ : \ g(da)=dg(a) \ \ {\rm
and }\ \  g(\vac)=0\ \}
\end{equation*}

\vskip .3cm

\noi and we take for $g\in C^1(V,M)$

\begin{equation*}
(\de_1 \, g)_z(a,b)=a\z g(b)-g(a\z b)+g(a)\z b
\end{equation*}

\vskip .4cm

\noi with $a,b\in V$. Hence, $H^1(V,M)=$Ker $\de_1=$VDer$(V,M)$.
We define $C^2(V,M)$ as the space of linear functions $f_z:V
\otimes V\to M((z))$ satisfying (for all $a,b\in V$)

\vskip.2cm

$\ast$
{\it Unit}:
\begin{align}\label{u}
f_z(a,\vac )=f_z(\vac , b)=0
\end{align}

$\ast$ {\it Translation - Derivation}:
\begin{align}\label{t}
 \frac{d}{dz}
f_z(a,b)=f_z(da,b),\qquad d(f_z(a,b))=f_z(da,b)+f_z(a,db), \
\end{align}

$\ast$
 {\it Symmetry}:
\begin{equation}\label{s}
f_z(a,b)=e^{zd}f_{-z}(b,a).
\end{equation}

\vskip .2cm

Now, let $Z^2(V,M)$ be the space of functions $f_z\in C^2(V,M)$
that satisfy that for all $a,b,c\in V$ there exists $n\in
\mathbb{N}$ such that
\begin{equation}\label{f-associa}
    (x+z)^n \Big[f_z(a\x b,c)+f_x(a,b)\z c \Big]=(x+z)^n \Big[a\xz
    f_z(b,c)+f_{x+z}(a,b\z c) \Big]
\end{equation}
and define $H^2(V,M)=Z^2(V,M)/$Im $\de_1$. If $V$ is an
associative commutative algebra, it is clear that we obtain the
Harrison's cohomology.

The associativity condition (\ref{associatt}) of a vertex algebra
produce the condition (\ref{f-associa}). But if we impose the
associator formula (\ref{associator}), we shall see after the
proof of Theorem \ref{h2-ext} that (\ref{f-associa}) could be
replaced by
\begin{align}\label{eee}
   0= f_z(a\x b,c)+f_x(a,b)\z c & -a\xz
    f_z(b,c)-f_{x+z}(a,b\z c)\\
    & - b\z(f_{z+x}(a,c)
    -f_{x+z}(a,c))-f_z(b,a\zx c- a\xz c).\nonumber
\end{align}
for all $a,b,c\in V$, and the RHS of (\ref{eee}) could be the
definition of $(\de_2 f)_{z,x}(a,b,c)$. In this case
$H^2(V,M)=$Ker $\de_2/$Im $\de_1$.

\begin{proposition}
$H^2(V,M)$ is well defined, that is Im $\de_1\subseteq Z^2(V,M)$.
\end{proposition}

\begin{proof}
Let $g:V\to M$ such that $dg(a)=g(da)$ and $g(\vac)=0$. We define
$f_z:V\otimes V\to M((z))$ by $f_z(a,b)=a\z g(b)-g(a\z b)+g(a)\z
b$. Now,
$$
f_z(\vac,b)=\vac\z g(b)-g(\vac\z b)+g(\vac)\z b=g(b)-g(b)=0,
$$
and
$$
f_z(a,\vac)=a\z g( \vac)-g(a\z \vac)+g(a)\z \vac=a\z
g(\vac)-g(e^{zd} a)+e^{zd} g(a)=0,
$$
therefore $f_z$ satisfies (\ref{u}). Now we prove that it
satisfies (\ref{t}):

\begin{align*}
\frac{d}{dz}f_z(a,b)=& da\z g(b)-g(da \z b)+d e^{zd}b\mz
g(a)-e^{zd}(db)\mz g(a)\\=& da\z g(b)-g(da \z b)+e^{zd}b\mz d
g(a)=f_z(da,b),
\end{align*}
and
\begin{align*}
    d(f_z(a,b))=& da\z g(b)+ a\z dg(b)-g(da\z b)-g(a\z
    db)+e^{zd}(db)\mz g(a)+e^{zd} b\mz dg(a) \\
    =& f_z(da,b)+f_z(a,db).
\end{align*}
The symmetry (\ref{s}) follows by
\begin{align*}
    e^{zd} f_{-z}(b,a)=& e^{zd} \, b\mz g(a)-e^{zd} g(b\mz
    a)+e^{zd}\left(e^{-zd} a\z g(b)\right)\\
    =& g(a)\z b-g(e^{zd} b\mz a)+a\z g(b)=f_z (a,b).
\end{align*}
Now, we should check (\ref{f-associa}):
\begin{align}\label{77}
f_z(a\x b,c)&+f_x(a,b)\z c - a\xz
    f_z(b,c)-f_{x+z}(a,b\z c)=\qquad \qquad \qquad \nonumber\\
    &=(a\x b)\z g(c)-a\xz(b\z
    g(c))
    -g((a\x b)\z c)+g(a\xz (b\z
    c))\\
    &\ \ \ \ -a\xz(g(b)\z c)
    +(a\x g(b))\z c +(g(a)\x b)\z c-g(a)\xz (b\z c)\nonumber
\end{align}
using the associativity (\ref{aaaa}) of $M$, the first two terms
in the RHS is zero after the multiplication by $(x+z)^n$ for some
$n\in \mathbb{N}$. Similarly with the third and fourth terms, by
using the associativity (\ref{associatt}) of $V$. Now, consider
the fifth and sixth terms in (\ref{77}). Using that $e^{zd} a\x
w=a\xz (e^{zd} w)$, we have
\begin{align*}
    (a\x g(b))\z c -a\xz(g(b)\z c)&=e^{zd} c\mz(a\x g(b))-a\xz
    (e^{zd} c\mz g(b))\\ &=e^{zd}[c\mz(a\x g(b))- a\x(c\mz g(b))]
\end{align*}
which is zero after the multiplication by $(x+z)^n$ for some $n\in
\mathbb{N}$, due to locality (\ref{locality}) in the action of $V$
on $M$.  Finally, consider the seventh and eighth terms in
(\ref{77}).
 On one hand, we have
 $$
 (g(a)\x b)\z c=e^{zd}c\mz (g(a)\x b)=e^{zd}c\mz(e^{xd} b\mx
 g(a))=e^{(z+x)d} c\mzmx (b\mx g(a)),
 $$
 and on the other hand, we have
 $$
 g(a)\xz(b\z c)=e^{xd}(b\z c)\mx (e^{zd} g(a))=e^{xd}(e^{zd} c
 \mz b)\mx(e^{zd} g(a))=e^{(x+z)d} (c\mz b)\mx g(a),
 $$
 then using the associativity (\ref{aaaa}), both terms are equal after the
 multiplication by $(x+z)^n$ for some $n\in \mathbb{N}$, finishing
 the proof.
\end{proof}

\

\vskip .5cm


\section{Second cohomology and square-zero extensions}\lbb{ex}


\

\begin{definition}
(a) Let $V$ be a vertex algebra. A {\it square-zero ideal of $V$}
is an ideal $W$ of $V$ such that for any $a,b\in W$, $a \x b=0$.

(b) Let $V$ be a vertex algebra and $M$ a  $V$-module. A {\it
square-zero extension $(\Lambda, f, g)$ of $V$ by $M$} is a vertex
algebra $\Lambda$ together with a surjective homomorphism $f:
\Lambda \to V$ of  vertex algebras such that $\ker f$ is a
square-zero ideal of $\Lambda$ (so that $\ker f$ has the structure
of a $V$-module) and an injective homomorphism $g$ of $V$-modules
from $M$ to $\Lambda$ such that $g(M)=\ker f$.

(c) Two square-zero extensions $(\Lambda_{1}, f_{1}, g_{1})$ and
$(\Lambda_{2}, f_{2}, g_{2})$ of $V$ by $M$ are {\it equivalent}
if there exists an isomorphism of vertex algebras $h:
\Lambda_{1}\to \Lambda_{2}$ such that the diagram
$$\begin{CD}
0@>>> M @>>g_{1}> \Lambda_{1} @>>f_{1}>V @>>>0\\
@. @V1_{M}VV @VhVV @VV1_{V}V\\
0@>>> M @>>g_{2}> \Lambda_{2} @>>f_{2}>V @>>>0,
\end{CD}$$
is commutative.
\end{definition}

Now, we have the following result:

\begin{theorem}\label{h2-ext}
Let $V$ be a  vertex algebra and $M$ a $V$-module. Then the set of
the equivalence classes of square-zero extensions of $V$ by $M$
corresponds bijectively to $H^{2}(V, M)$.
\end{theorem}

\begin{proof}
Let $(\Lambda, f, g)$  be a square-zero extension of $V$ by $M$.
Then there is an injective linear map  $\Gamma: V\to \Lambda$ such
that the linear map $h: V\oplus M\to \Lambda$ given by $h(a,
u)=\Gamma(a)+g(u)$ is a linear isomorphism. By definition, the
restriction of $h$ to $M$ is the isomorphism $g$ from $M$ to $\ker
f$. Then the vertex algebra structure and the $V$-module structure
on $\Lambda$ give a  vertex algebra structure and a $V$-module
structure on $V\oplus M$ such that the embedding $i_{2}: M\to
V\oplus M$  and the projection $p_{1}: V\oplus M\to V$ are
homomorphisms of vertex algebras. Moreover, $\ker p_{1}$ is a
square-zero ideal of $V\oplus M$, $i_{2}$ is an injective
homomorphism such that $i_{2}(M)=\ker p_{1}$ and the diagram
\begin{equation*}
\begin{CD}
0@>>> M @>i_{2}>> V\oplus M @>p_{1}>>V @>>>0\\
@. @V1_{M}VV @VhVV @VV1_{V}V\\
0@>>> M @>>g> \Lambda @>>f>V @>>>0
\end{CD}
\end{equation*}
of $V$-modules is commutative. So we obtain a  square-zero
extension $(V\oplus M, p_{1}, i_{2})$ equivalent to $(\Lambda, f,
g)$. We need only consider square-zero extension of $V$ by $M$ of
the particular form $(V\oplus M, p_{1}, i_{2})$. Note that the
difference between two such square-zero extensions are in the
$\z$-product  maps. So we use $(V\oplus M, \VMz , p_{1}, i_{2})$
to denote such a square-zero extension.

We now write down the $\z$-product map for $V\oplus M$ explicitly.
Since $(V\oplus M, \VMz, p_{1}, i_{2})$ is a square-zero extension
of $V$, there exists a linear map $\psi_z:V\otimes V\to M((z))$
such that

\begin{eqnarray}\label{Y_V+W}
(a,u)\VMz (b, v)=(a\z b, a\z v+ u\z b+\psi_z(a,b))
\end{eqnarray}
for $a,b\in V$ and $u,v\in M$.

Now, we shall prove that $V\oplus M$ with $\VMz$, the vacuum
vector $\vac_{V\oplus M}=(\vac,0)$ and $d_{V\oplus M}(a,u)=(d_V a,
d_M u)$, is a vertex algebra if and only if $\psi_z \in Z^2(V,M)$.
In order to simplify the proof, observe that in Proposition 4.8.1
in \cite{LL}, they showed that $V\oplus M$ with $\vac_{V\oplus
M}$, $d_{V\oplus M}$ and $\VMz$ corresponding to $\psi_z\equiv 0$,
is a vertex algebra. Therefore, when we check the axioms, we know
that all the terms without $\psi_z$ satisfy the corresponding
equation. So, in order to prove that $V\oplus M$ with $\VMz$ given
by (\ref{Y_V+W}) is a vertex algebra, we only need to see the
terms with $\psi_z$. For example, the element $(\one, 0)$
satisfies (for $a,b\in V$ and $u,v\in M$)
\begin{equation*}
(\vac ,0)\VMz (b,v)=(\vac\z b,\vac\z v +\psi_z(\vac,b))=(b,v)
\end{equation*}
and
\begin{equation*}
(a ,u)\VMz (\vac,0)=(a\z \vac,u\z \vac +\psi_z(a,\vac))=(e^{zd_V}
a, e^{zd_M} (\vac\mz u) + \psi_z(a,\vac))=e^{zd_{V\oplus M}}(a,u)
\end{equation*}
if and only if $\psi_z(\vac ,b)=0=\psi_z(a,\vac)$ for all $a,b\in
V$. From now on, we use $d=d_V=d_M$. A simple computation shows
that $V\oplus M$ satisfies the translation-derivation properties
if and only if for all $a,b\in V$
\begin{align*}
 \frac{d}{dz}
\psi_z(a,b)=\psi_z(da,b),\qquad \hbox{ and }\ \
d(\psi_z(a,b))=\psi_z(da,b)+\psi_z(a,db).
\end{align*}
%

Now, consider the commutativity axiom, that is:
\begin{align*}
    e^{zd_{V \oplus M}}(b,v)\VMmz (a,u)&=
    e^{zd_{V \oplus M}}(b\mz a, b\mz u+e^{-zd} a\z
    v+\psi_{-z}(b,a))\\
    &=(e^{zd} b\mz a,e^{zd} b\mz
    u+a\z v+e^{zd}\psi_{-z}(b,a))
\end{align*}
and $(a,u)\VMz (b,v)=(a\z b, a\z v+e^{zd}b\mz u+\psi_z(a,b))$.
Therefore, $\psi_z$ must satisfy $\psi_z(a,b)=e^{zd}
\psi_{-z}(b,a)$.

Similarly, expanding
$$
\Big((a,u)\VMx(b,v)\Big)\VMz (c,w) \qquad \hbox{ and } \qquad
(a,u)\VMxz \Big((b,v)\VMz(c,w)\Big),
$$
and taking the terms with $\psi_z$, it is easy to see that the
associativity axiom (\ref{associatt}) holds if and only if for all
$a,b,c\in V$ there exists $n\in \mathbb{N}$ such that
\begin{equation}\label{f-associal}
    (x+z)^n \Big[\psi_z(a\x b,c)+\psi_x(a,b)\z c \Big]=(x+z)^n \Big[a\xz
    \psi_z(b,c)+\psi_{x+z}(a,b\z c) \Big],
\end{equation}
proving that $V\oplus M$ is a vertex algebra if and only if
$\psi_z\in Z^2(V,M)$, and together with the projection
$p_1:V\oplus M\to V$ and the embedding $i_2:M\to V\oplus M$,
$V\oplus M$ is a square-zero extension of $V$ by $M$.

Next we prove that two elements of $Z^2(V,M)$ obtained in this way
differ by an element of $\delta_{1}C^{1}(V, M)$ if and only if the
corresponding square-zero extensions of $V$ by $M$ are equivalent.

Let $\psi, \phi \in Z^2(V,M)$ be two such elements obtained from
square-zero extensions $(V\oplus M, \zu, p_{1}, i_{2})$ and
$(V\oplus M, \zd , p_{1}, i_{2})$, respectively. Assume that
$\psi=\phi+\delta_{1}(g)$ where $g \in C^{1}(V, M)$.

We now define a linear map $h: V\oplus M\to V\oplus M$ by
$$h(a, u)=(a, u+g(a))$$
for $a\in V$ and $u\in M$. Then $h$ is a linear isomorphism and it
satisfies (for $a,b\in V$ and $u,v\in M$)
\begin{align}\label{hhh}
    h\left((a,u)\zu (b,v)\right)&=(a\z b,a\z v+u\z b+\psi_z(a,b)+g(a\z b))
    \nonumber\\
    &=(a\z b,a\z v+u\z b+\phi_z(a,b)+a\z g(b)+g(a)\z b)\\
    &=(a, u+g(a))\zd (b,v+g(b))\nonumber\\
    &=h(a,u)\zd h(b,v).\nonumber
\end{align}
Thus $h$ is an isomorphism of vertex algebras from $(V\oplus M,
\zu, (\vac, 0))$ to $(V\oplus M, \zd , (\vac, 0))$ such that the
diagram
\begin{equation}\label{equivalence}
\begin{CD}
0@>>> M @>i_{2}>> V\oplus M @>p_{1}>>V @>>>0\\
@. @V1_{M}VV @VhVV @VV1_{V}V\\
0@>>> M @>i_{2}>> V\oplus M @>p_{1}>>V @>>>0\\
\end{CD}
\end{equation}
is commutative. Thus the two square-zero extensions of $V$ by $M$
are equivalent.

Conversely, let $(V\oplus M, \zu , p_{1}, i_{2})$ and $(V\oplus M,
\zd , p_{1}, i_{2})$ be two equivalent square-zero extensions of
$V$ by $M$. So there exists an isomorphism $h: V\oplus M\to
V\oplus M$ of vertex algebras such that (\ref{equivalence}) is
commutative. Let $h(a, u)=(f(a, u), g(a, u))$ for $a\in V$ and
$u\in M$. Then by (\ref{equivalence}), we have $f(a, u)=a$ and
$g(0, u)=u$. Since $h$ is linear, we have $g(a, u)=g(a, 0)+g(0,
u)=u+g(a, 0)$. So $h(a, u)=(a, u+g(a, 0))$. Taking $g(a)$ to be
$g(a, 0)$, we see that there exists a linear map $g: V\to M$ such
that $h(a, u)=(a, u+g(a))$. Using that $d\,h(a,0)=h(d(a,0))$ and
$h(\vac,0)=(\vac,0)$, it is clear that $g(da)=dg(a)$ and
$g(\vac)=0$. Thus, $g\in C^1(V,M)$.

Let $\psi$ and $\phi$ be elements of $Z^2(V,M)$ obtained from
$(V\oplus M, \zu , p_{1}, i_{2})$ and $(V\oplus M, \zd , p_{1},
i_{2})$, respectively. Then, since $h$ is a homomorphism of vertex
algebras, (\ref{hhh}) holds for $a,b\in V$ and $u,v \in M$.  So
the two expressions in the middle of (\ref{hhh}) are equal. Thus,
we have $\psi=\phi+\de_1(g)$. Therefore, $\psi$ and $\phi$ differ
by an element of $\delta_{1}C^{1}(V, M)$.
\end{proof}

\

In the proof of Theorem \ref{h2-ext} (by taking the terms with
$\psi_z$),  we saw in (\ref{f-associal}) that the associativity
axiom holds in $(V\oplus M, \VMz, (\vac,0))$ if and only if for
all $a,b,c\in V$ there exists $n\in \mathbb{N}$ such that
\begin{equation*}
    (x+z)^n \Big[\psi_z(a\x b,c)+\psi_x(a,b)\z c \Big]=(x+z)^n \Big[a\xz
    \psi_z(b,c)+\psi_{x+z}(a,b\z c) \Big].
\end{equation*}

Recall that  we can replace the associativity axiom
(\ref{associatt}) in the definition of vertex algebra, by the
associator formula (\ref{associator}). By taking the terms with
$\psi_z$,  it is possible to prove  that the associator formula
holds in $(V\oplus M, \VMz, (\vac,0))$ if and only if for all
$a,b,c\in V$ we have
\begin{align*}
   0= \psi_z(a\x b,c)+\psi_x(a,b)\z c & -a\xz
    \psi_z(b,c)-\psi_{x+z}(a,b\z c)\\
    & - b\z(\psi_{z+x}(a,c)
    -\psi_{x+z}(a,c))-\psi_z(b,a\zx c- a\xz c).\nonumber
\end{align*}

We avoid replacing the associativity or associator formula by the
Jacobi identity because we want to make emphasis that a vertex
algebra is a generalization of an associative commutative algebra,
and that the cohomology must be a generalization of Harrison
cohomology.

\vskip .3cm

\

\section{Second cohomology and first order deformations}

\

\begin{definition}
(a) Let $t$ be a formal variable and let $(V,\z , \one, d)$ be a
 vertex algebra. A {\it first order deformation of $V$} is
  a family of $z$-products of the
form
\begin{equation*}
    a\sz b= a\z b + t \, f_z(a,b)
\end{equation*}
with $a,b\in V$, where $f_z:V\otimes V \to V((z))$  is a linear
map (independent of $t$), such that $(V,\szz , \vac,d)$ is a family
of vertex algebras up to the first order in $t$ (ie. modulo
$t^2$). More precisely, the quadruple $(V, \szz, \one, d)$
satisfies the following conditions:

\vskip.1cm

$\ast$ {\it Unit}:
\begin{equation}\label{uu}
 \vac \sz a= a \quad  \textrm{ and } \quad a \sz \vac=e^{zd}a;
\end{equation}

$\ast$ {\it Translation - Derivation}:
\begin{align}\label{dd}
(da)\sz b= \ddz (a\sz b), \quad d(a\sz b)=(da)\sz b + a\sz (db);
\end{align}

$\ast$ {\it Commutativity}:
\begin{align} \label{ssss}
a \sz b= e^{zd} (b\smz a);
\end{align}

$\ast$  {\it Associativity up to the first order in $t$}: For any
$a,b,c\in V$, there exist $l\in \NN$ such that

\begin{align} \label{aa}
(z+w)^l\, \ (a \sz b) \sw c = (z+w)^l\, \  a\, \szw \, (b\sw c) \quad
\hbox{ mod } t^2.
\end{align}
\vskip.3cm

(b) Two first order deformations $\szu$ and $\szd$ of $(V, \z ,
\one,d)$ are {\it equivalent} if there exists a family $\phi_{t}:
V \to V[t]$,  of linear maps of the form $\phi_{t}=1_{V}+tg$ where
$g: V\to V$ is a linear map such that
\begin{equation*}
\phi_{t}(a\szu b)=\phi_{t}(a)\szd \phi_{t}(b) \quad \hbox{ mod }
t^2
\end{equation*}
for $a,b \in V$.
\end{definition}

We have:

\begin{theorem}\label{deform}
Let $V$ be a vertex algebra. Then the set of the equivalence
classes of first order deformations of  $V$ corresponds
bijectively to $H^{2}(V, V)$.
\end{theorem}

\begin{proof}
Let $\szz$ be a first order deformation of $V$. By definition,
there exists a linear map $f_z:V\otimes V\to V((z))$ such that
\begin{equation}\label{ss}
    a\sz b=a\z b+ t \, f_z(a,b)
\end{equation}
for $a,b\in V$, and $(V,\szz ,\vac, d)$ is a family of vertex
algebras up to the first order in $t$.

The unit properties (\ref{uu}) for $(V,\szz ,\vac, d)$ gives
$$
\vac \sz a=\vac\z a+t\, f_z(\vac,a)=a
$$
and
$$
a\sz \vac=a\z \vac+t\, f_z(a,\vac)=e^{zd} a
$$
for $a\in V$. So, they are equivalent to
\begin{equation}\label{11}
    f_z(a,\vac )=0=f_z(\vac, a) \quad\hbox{ for all } a\in V.
\end{equation}
Similarly, the coefficient of $t^{0}$ of the
Translation-Derivation properties (\ref{dd}) corresponds exactly
to the Translation-Derivation properties of $(V,\z ,\vac, d)$, and
the coefficient of $t^{1}$ in (\ref{dd}) corresponds exactly to
the following properties on $f_z$:
\begin{equation}\label{22}
    \ddz f_z(a,b)=f_z(d a,b), \quad \hbox{ and }\quad
    d(f_z(a,b))=f_z(da,b)+f_z(a,db).
\end{equation}
Now, the coefficient of $t^{0}$ of the Commutativity property
(\ref{ssss}) corresponds exactly to the Commu- tativity property of
$(V,\z ,\vac, d)$, and the coefficient of $t^{1}$ in (\ref{ssss})
corresponds exactly to the following property on $f_z$:
\begin{equation}\label{4444}
    f_z(a,b)=e^{zd} f_{-z}(b,a).
\end{equation}
In the same way, using (\ref{ss}), we take the expansions modulo
$t^2$ of the expressions
$$
(a \sz b) \sw c \quad \hbox{ and }\quad \  a\, \szw \, (b\sw c)
$$
and we consider the coefficients of $t^{0}$ and $t^{1}$ of them.
By a direct computation, we can see that the coefficient of $t^0$
of the associativity property (\ref{aa}) corresponds exactly to
the associativity property of $\z$, and the coefficient of $t^1$
corresponds exactly to the following property: for all $a,b,c\in
V$ there exists $n\in\mathbb{N}$ such that
\begin{equation}\label{33}
    (w+z)^n \Big[f_w(a\z b,c)+f_z(a,b)\w c \Big]=(w+z)^n \Big[a\zw
    f_w(b,c)+f_{z+w}(a,b\w c) \Big].
\end{equation}
Therefore, using (\ref{11}), (\ref{22}),(\ref{4444}) and
(\ref{33}), we have seen that
$$
a\sz b=a\z b+t\, f_z(a,b)
$$
is a first order deformation of $V$
if and only if $f_z\in Z^2(V,V)$.

Now, we prove that two first order deformations of $V$ are
equivalent if and only if the difference between the corresponding
elements in $Z^2(V,V)$ is in Im $\de_1$.

Consider two first order deformations of $V$ given by
$$
a\szu b=a\z b+t\,\psi_z(a,b)\quad \hbox{ and }\quad a\szd b=a\z
b+\phi_z(a,b),
$$
where $\psi_z$ and $\phi_z$ are in $Z^2(V,V)$. They are equivalent
if and only if there exists $f_t=1_V+t\, g$ where $g:V\to V$ is a
linear map such that
\begin{equation}\label{44}
f_{t}(a\szu b)=f_{t}(a)\szd f_{t}(b) \quad \hbox{ mod } t^2
\end{equation}
for $a,b \in V$. Now, since
$$
f_t(a\szu b)=f_t(a\z b+t\,\psi_z(a,b))=a\z b+t\,\psi_z(a,b)+t\,
g(a\z b) \quad \hbox{ mod } t^2
$$
and
$$
f_t(a)\szd f_t (b)=(a+t\,g(a))\szd (b+t\,g(b))= a\z b +t\, a\z
g(b)+t\,g(a)\z b+t\, \phi_z(a,b) \quad \hbox{ mod } t^2
$$
then (\ref{44}) is equivalent to
$$
\psi_z(a,b)-\phi_z(a,b)=a\z g(b)-g(a\z b)+g(a)\z b
$$
for all $a,b\in V$. Therefore, it is equivalent to
$\psi_z-\phi_z=(\de_1 g)_z$, finishing the proof.
\end{proof}

\vskip .5cm

\vskip .5cm

\subsection*{Acknowledgements}
The author would like to thank M. V. Postiguillo and C. Bortni for
their constant help and support throughout this work. Special
thanks to C. Boyallian.

\bibliographystyle{amsalpha}


\end{document}